\documentclass[letterpaper]{amsart}
\usepackage[utf8]{inputenc}
\usepackage{amsfonts,amssymb,amsbsy,amsmath} 
\usepackage{bm} 
\usepackage[inline,shortlabels]{enumitem} 
\usepackage{listings} 
\usepackage{amsthm} 
\usepackage{commath} 
\usepackage{booktabs} 
\usepackage[letterpaper, margin=1in]{geometry}
\usepackage{relsize}
\usepackage{tikz-cd}
\usepackage{tikz}
\usepackage{booktabs}
\usepackage{verbatim}
\usepackage{algorithm,algpseudocode}
\usepackage{todonotes}
\usepackage{stmaryrd}
\usepackage{mathtools}
\usepackage{hyperref}
\usepackage{cleveref} 
\usepackage{kbordermatrix}
\usetikzlibrary{patterns}

\theoremstyle{plain}
\newtheorem{theorem}{Theorem}[section]
\newtheorem{lemma}[theorem]{Lemma}
\newtheorem{corollary}[theorem]{Corollary}
\newtheorem{proposition}[theorem]{Proposition}

\theoremstyle{definition}
\newtheorem{definition}[theorem]{Definition}
\newtheorem{example}[theorem]{Example}

\newtheorem{remark}[theorem]{Remark}

\newcommand{\B}[1]{\mathbb #1}

\DeclareMathOperator{\spn}{span}
\newcommand{\ev}{\mathrm{ev}}

\newcommand{\p}{\partial}
\newcommand{\ideal}[1]{( #1 )}
\newcommand{\Span}[1]{\langle #1 \rangle}
\newcommand*{\dashedRightArrow}{\mathbin{\tikz [baseline=-0.25ex,-latex, dashed] \draw [->,densely dashed] (0pt,0.5ex) -- (1.3em,0.5ex);}}

\newcommand{\QQ}{{\B Q}}
\newcommand{\CC}{{\B C}}
\newcommand{\VV}{{\B V}}
\newcommand{\KK}{{\B K}}

\newcommand{\FF}{{\B F}}
\newcommand{\DD}{\mathcal D}
\newcommand{\mm}{\mathfrak m}
\newcommand{\ZZp}{\B N}
\newcommand{\NN}{\B N}
\newcommand{\cl}[1]{\overline{#1}}

\newcommand{\blank}{\underline{\phantom{n}}}
\newcommand{\xx}{\boldsymbol{x}}
\renewcommand{\tt}{\boldsymbol{t}}

\newcommand{\setOfNOs}{set of Noetherian operators}

\algrenewcommand\algorithmicrequire{\textbf{Input}}
\algrenewcommand\algorithmicensure{\textbf{Output}}

\allowdisplaybreaks

\title{Noetherian operators and primary decomposition}
\author{Justin Chen, Marc H\"ark\"onen, Robert Krone, Anton Leykin}

\address{School of Mathematics, Georgia Institute of Technology,
Atlanta, Georgia}
\email{\{justin.chen,leykin\}@math.gatech.edu}
\email{harkonen@gatech.edu}
\thanks{Research of MH and AL is supported in part by NSF DMS-1719968. MH is partially supported by the Vilho, Yrj\"o and Kalle V\"ais\"al\"a Foundation}

\address{Department of Mathematics, University of California, Davis, California}
\email{rkrone@math.ucdavis.edu}


\begin{document}
\maketitle

\begin{abstract}
Noetherian operators are differential operators that encode primary components of a polynomial ideal. 
We develop a framework, as well as algorithms, for computing Noetherian operators with local dual spaces, both symbolically and numerically. 
For a primary ideal, such operators provide an  alternative representation to one given by a set of generators. 
This description fits well with numerical algebraic geometry, taking a step toward the goal of numerical primary decomposition.
\end{abstract}

\section{Introduction} 
\label{sec:intro}
A fundamental problem in computational algebra is \emph{primary decomposition}: given an ideal, find the associated primes, and express the ideal as an intersection of primary components. 
When the ideal is radical, this corresponds geometrically to decomposing an algebraic variety into a union of irreducible components.

Algorithms implemented in computer algebra systems (\cite{gianni1988grobner}, \cite{shimoyama1996localization}, \cite{eisenbud1992direct}, \cite{decker1999primary}) perform primary decomposition for ideals in polynomial rings by producing a set of ideal generators for each primary component.
Although providing generators is the most direct way to represent a primary ideal, in practice it is often infeasible to compute primary decomposition this way, e.g. due to the size of the generators. 
Thus it makes sense to seek an alternative approach to primary decomposition which can harness the power of numerical methods. 
The natural setting for this is {\em numerical algebraic geometry}~\cite{Sommese-Wampler-Verschelde-intro,Sommese-Wampler-book}, which provides a suite of algorithms for computing with complex algebraic varieties using numerical techniques.
For certain tasks, numerical methods may solve problems that are difficult for typical symbolic methods. 
As an example, {numerical irreducible decomposition}~\cite{sommese2001numerical}  has been used to decompose varieties that were outside the feasible range of symbolic algorithms; see for instance \cite{bates2011toward,hauenstein2018exceptional}.

In contrast to the description by a set of generators, a primary ideal $I$ can be described by two pieces of data: 
its minimal prime $\sqrt{I}$ (or geometrically, the variety $\VV(I)$), and the \emph{multiplicity structure} of $I$ over $\sqrt{I}$. 
One can describe the multiplicity structure of $I$ via associated differential operators on $\VV(I)$:

\begin{definition}\label{def:noethOps}
    A set $N$ of differential operators with polynomial coefficients is called a set of Noetherian operators for $I$ if $f \in I \iff D \bullet f \in \sqrt{I} \; \forall D \in N$.
\end{definition}

The idea of representing a (primary) ideal in a polynomial ring via a dual set of differential operators is both natural and classical, dating back to Macaulay (who introduced inverse systems in \cite{macaulay1994algebraic}) and Gr\"obner \cite{grobner1938neue}. 
Since their introduction by Palamodov in 1970 \cite{palamodov1970linear}, Noetherian operators have been sporadically studied in the literature: \cite{oberst1999construction}, \cite{sturmfels2002solving}, \cite{nonkan2013weyl}, and \cite{yairon_roser_bernd}.
Symbolic algorithms to compute Noetherian operators were developed and implemented in \cite{damiano} and \cite{yairon_roser_bernd}.

\medskip

Our contribution consists of new algorithms to compute a set of Noetherian operators representing a primary ideal, as well as theoretical results leading up to them. 
We develop two algorithms: one using exact symbolic computation (\Cref{alg:symb_zero_dim}) and the other based on hybrid symbolic-numeric methods of numerical algebraic geometry (\Cref{alg:mainAlgorithm}).

Our symbolic algorithm follows a path started by Macaulay \cite{macaulay1994algebraic} reducing the problem to linear algebra. 
The potential of this approach is that the body of work in this direction may be adapted to computation of Noetherian operators; for instance, optimizations of the algorithm as in \cite{Mourrain:inverse-systems} are possible.

Our numerical algorithm may solve problems that are out of reach for purely symbolic techniques (cf. e.g. \Cref{ex:carpet}). 
Given an ideal with no embedded components, our numerical algorithm combined with numerical irreducible decomposition leads to \emph{numerical primary decomposition} (\Cref{alg:num_primary_decomp}): i.e. a numerical description of all components of the ideal, which e.g. enables a probabilistic membership test. 

Numerical irreducible decomposition algorithms are efficient but set-theoretic in nature. 
In contrast, numerical detection of embedded components, studied in \cite{leykin2008numerical,krone2017numerical}, is a rather difficult task. 
Our procedures for describing primary components via Noetherian operators assumes that the associated primes of the ideal have already been discovered. 
Moreover, our algorithms rely on primes being isolated, i.e. not embedded (see \Cref{rmk:isol}) and therefore do not address the problem of finding embedded components.
Nevertheless, it will be interesting to study in the future whether Noetherian operators can make a contribution here.

\medskip

The paper is organized as follows. 
\Cref{sec:preliminaries} gives a gentle introduction to Noetherian operators and classical dual spaces.
\Cref{sec:macaulay_matrices} generalizes the definition of a dual space to nonrational points and develops theory that leads to a symbolic algorithm based on Macaulay matrices. 
\Cref{sec:numerical_approach} deals with specialization and interpolation of Noetherian operators, leading to a numerical algorithm for computing Noetherian operators as well as an algorithm for numerical primary decomposition. \Cref{sec:gen_properties} concludes with general properties of Noetherian operators for non-primary ideals. 

Algorithms are implemented in Macaulay2 \cite{grayson2002macaulay2}, and the software is available on GitHub\footnote{\texttt{NoetherianOperators} codebase: \url{https://github.com/haerski/NoetherianOperators}.}.

\section{Preliminaries}\label{sec:preliminaries}
Let 
$\KK$ be a field of characteristic $0$ and 
$R := \KK[\xx] = \KK[x_1,\dotsc,x_n]$ a polynomial ring over $\KK$.
For numerical applications our focus will be on the case $\KK=\CC$, as implementations of numerical methods generally use floating point approximations of complex numbers to some fixed precision. 
On the other hand, our symbolic algorithms do not assume that $\KK$ is algebraically closed. We often take $\KK=\QQ$ in examples.

\subsection{Sets of Noetherian operators} \label{sec:NOsets}
In \Cref{def:noethOps}
we consider an ideal $I \subseteq R$ and a set $N$ of differential operators in 
$$W_R := R \langle \partial_1, \ldots, \partial_n \rangle, \quad\text{where } \partial_i := \frac{\partial}{\partial x_i},$$ 
the noncommutative ring of differential operators with coefficients in $R$, known as the $n$-dimensional \emph{Weyl algebra} over $R$.  The differential operators $\partial_1,\ldots, \partial_n$ are $\KK$-linear endomorphisms of $R$ satisfying the relations $\partial_i x_j - x_j \partial_i = \delta_{ij}$.

\begin{remark}[Ideal membership test]
Let $\B V(I) \subseteq \KK^n$ be the affine variety defined by $I$.
A set of Noetherian operators $N = \{D_1,\ldots,D_r\}$ for $I$ gives a probabilistic test for determining if a polynomial $f$ is in $I$ or not, assuming an oracle for sampling a random point from $\B V(I)$ (according to some reasonable distribution).
The set $\{D_1 \bullet f, \ldots, D_r\bullet f\}$ is contained in $\sqrt{I}$ if and only if $f \in I$.
If $p \in \B V(I)$ is general, then $(D_i\bullet f)(p)$ evaluates to zero for all $i = 1,\ldots,r$ if and only if $f \in I$.
\end{remark}

If $I = \sqrt{I}$ is radical, then the singleton $\{1\}$ is a set of Noetherian operators for $I$.
The case of most interest is when $I$ is primary, but not radical.
In this case, a minimal set of Noetherian operators for $I$ has more than one element.  Although such a set need not be unique, its cardinality equals the multiplicity of $I$ over $\sqrt{I}$, which is the ratio $e(I)/e(\sqrt{I})$ of their Hilbert-Samuel multiplicities, see the proof of \Cref{thm:numerical_noetherian_operators}. 

\begin{example}
Let $I = \ideal{(x+y+1)^m} \subseteq \KK[x,y]$, a primary ideal. Then the sets $N_1 = \{1,\p_x, \dots, \p_x^{m-1}\}$ and $N_2 = \{1,\p_y, \dots, \p_y^{m-1}\}$
are both minimal sets of Noetherian operators for $I$.

Note that the generator of $I$ in expanded form consists of $\binom{m+2}{2}$ monomials with integer coefficients that grow with $m$. On the other hand, both $N_1$ and $N_2$ are much simpler expressions of size $m$; moreover, either of them, together with the radical $\sqrt I = \ideal{x+y+1}$, describes the ideal $I$ fully.       
\end{example}

For our numerical algorithm one may not even have generators for the radical of $I$ available, which is the case in~Example~\ref{ex:carpet}. 
Moreover, the input can be a set of generators of an ideal (for which $I$ is a component) which are only available as black-box differentiable evaluation routines. 
We mention this here in order to preempt the common assumption in classical computational algebraic geometry that polynomials are always represented as sums of their monomial terms.

\subsection{Dual spaces}
We start by reviewing the classical theory of Macaulay dual spaces.
The \emph{dual space} $R^*$ is by definition the $\KK$-vector space dual of $R$, i.e. the $\KK$-vector space of $\KK$-linear functionals $R \to \KK$. 
Let $p = (p_1, \ldots, p_n) \in \KK^n$ be a $\KK$-rational point. 
The polynomials $\{(\xx-p)^\alpha \coloneqq (x_1 - p_1)^{\alpha_1} \dotsm (x_n - p_n)^{\alpha_n}\}_{\alpha \in \ZZp^n}$ form a $\KK$-basis of $R$. 
Let $\ev_p : R \to \KK$ denote the evaluation functional at $p$, and $\mm_p := (x_1 - p_1, \ldots, x_n - p_n)$ the maximal ideal in $R$ associated to the point $p$.
Note that $\ev_p$ coincides with the natural surjection $R \twoheadrightarrow R/\mm_p \cong \KK$.

Post-composing differential operators with the evaluation functional produces new functionals. 
Let $\partial_{p,i}$ denote the functional $\ev_p \circ \partial_i$, and for a multi-index $\alpha = (\alpha_1,\dotsc,\alpha_n) \in \ZZp^n$ let
\begin{align*}
	\partial_p^\alpha : R &\rightarrow \KK\\
	f &\mapsto (\ev_p \circ \partial_1^{\alpha_1} \circ \dotsb \circ \partial_n^{\alpha_n})(f).
\end{align*}

The elements of $R^*$ can be expressed as formal power series in the $\partial_{p,i}$, and we write $R^* := \KK\llbracket \partial_p \rrbracket = \KK\llbracket \partial_{p,1}, \dotsc, \partial_{p,n} \rrbracket$. 
The $\KK$-linear span of $\{\partial_p^\alpha\}_{\alpha \in \ZZp^n}$ will be denoted $\KK[\partial_p]$.

\begin{definition}
	Let $I \subseteq R$ be an ideal. 
	The \emph{orthogonal complement of} $I$ is the $\KK$-vector subspace of $R^*$
	\begin{align*}
		I^\perp \coloneqq \{D \in R^* \mid D(f) = 0 \text{ for all } f\in I \}.
	\end{align*}
	If $\mathcal{D}$ is a $\KK$-vector subspace of $R^*$, then the \emph{orthogonal complement of} $\mathcal{D}$ is the $\KK$-vector subspace of $R$
	\begin{align*}
		\mathcal{D}^\perp \coloneqq \{f \in R \mid D(f) = 0 \text{ for all } D\in \mathcal D\}.
	\end{align*}
\end{definition}

\begin{proposition} \label{prop:perp_properties}
	For any ideals $I,J \subseteq R$ and any $\KK$-vector subspace $\mathcal D \subseteq R^*$, we have:
	\begin{enumerate}
			\item $I \subseteq J \iff I^\perp \supseteq J^\perp$
			\item $(I \cap J)^{\perp} = I^\perp + J^\perp$
			\item $(I+J)^\perp = I^\perp \cap J^\perp$
			\item $I^{\perp \perp} = I$, $\mathcal D^{\perp \perp} = \mathcal D$.
		\end{enumerate}
\end{proposition}

Note that $R^*$ has a natural $R$-module structure given by
\begin{align*}
	f \cdot \Lambda : R &\to \KK\\
	g &\mapsto \Lambda(fg)
\end{align*}
for $f,g\in R$, $\Lambda \in R^*$. The basis $\{(\xx-p)^\alpha\}_\alpha$ of $R$ acts on $\{\partial_p^\alpha\}_\alpha \subseteq R^*$ in the following way:
\begin{align*}
 	(x_i - p_i) \cdot \partial_p^\alpha = 
		\alpha_i \partial_{p,1}^{\alpha_1} \dotsm \partial_{p,i}^{\alpha_{i} - 1} \dotsm \partial_{p,n}^{\alpha_n}
\end{align*} 
We say that a $\KK$-subspace $\DD \subseteq R^*$ is \emph{closed} under the $R$-action if $\DD$ is an $R$-submodule of $R^*$. 
In general, $\DD$ is an $R$-submodule of $R^*$ iff $\DD^\perp$ is an $R$-submodule of $R$, i.e. an ideal of $R$.

\section{A symbolic approach via Macaulay matrices} 
\label{sec:macaulay_matrices}

\subsection{Dual spaces at nonrational points}\label{sec:dual-spaces-for-nonrational-points}
Next, we provide a generalization of dual spaces for nonrational points. 
For any $R$-algebra $A$, set $W_A := A \otimes_R W_R$, where $W_R := R\langle \partial_1,\dotsc,\partial_n \rangle$ is the Weyl algebra over $R$ (as in \Cref{sec:NOsets}). 
There is a natural action $\blank \bullet \blank \colon W_R \times R \to R$ given by $x_i \bullet f = x_i f$ and $\partial_i \bullet f = \frac{\partial f}{\partial x_i}$, which induces a natural $\KK$-bilinear pairing
\begin{align} \label{def:pairing}
	\langle \cdot , \cdot \rangle_A \colon W_{A} \times R \to A
\end{align}
This pairing is $A$-linear in the first argument, and makes the diagram
\[
		\begin{tikzcd}
			W_R \times R \arrow[r, "\blank \bullet \blank"] \arrow[d]
			& R \arrow[d] \\
			W_{A} \times R \arrow[r, "{\langle \cdot , \cdot \rangle_A}"]
			& A
		\end{tikzcd}
\]
commute. 
It can be viewed as follows: for any $f \in R$, $\langle \partial_i, f \rangle_A$ is the image of $\frac{\partial f}{\partial x_i}$ in $A$. 
We will often omit the subscript in $\langle \cdot,\cdot \rangle_A$ when $A$ is clear from context. 

\begin{definition}\label{def:local_dual_space}
	Let $I \subseteq R$ be an $R$-ideal, and $P \subseteq R$ a prime ideal with residue field $\kappa(P)$. 
	The \emph{local dual space of $I$ at $P$}, denoted $D_P[I]$, is the $\KK$-vector subspace orthogonal to $I$ with respect to the pairing $\langle \cdot, \cdot \rangle_{\kappa(P)}$, that is
	\begin{align*}
		D_P[I] := \{D \in W_{\kappa(P)} \mid \langle D,f \rangle_{\kappa(P)} = 0 \text{ for all } f \in I\}.
	\end{align*}
\end{definition}
This generalizes the definition of local dual spaces in \cite{krone2017eliminating}. 
Note that by $\kappa(P)$-linearity of $\langle \cdot, \cdot \rangle_{\kappa(P)}$ in the first argument, $D_P[I]$ is also a $\kappa(P)$-vector space. 
We call the $R$-module action induced by the $\kappa(P)$-vector space structure on $D_P[I]$ the left $R$-module action. 
We can also define a right $R$-module action on $D_P[I]$ analogous to the $R$-action on $I^\perp$, namely via
\begin{align*}
	\langle D \cdot f, g \rangle = \langle D, fg \rangle
\end{align*}
for $D\in D_P[I]$ and $f,g \in R$. 
These actions give rise to a natural $R$-bimodule structure on $D_P[I]$, cf. \cite{yairon_roser_bernd} for a treatment along these lines. 
Analogous to \Cref{prop:perp_properties}, one has:

\begin{proposition}\label{prop:local_dual_properties}
    Let $I,J \subseteq R$ be ideals and $P \subseteq R$ a prime. Then
    \begin{enumerate}
        \item If $I \subseteq J$, then $D_P[I] \supseteq D_P[J]$
        \item $D_P[I+J] = D_P[I] \cap D_P[J]$
        \item $D_P[I\cap J] = D_P[I] + D_P[J]$.
    \end{enumerate}
\end{proposition}

\begin{remark}
When the prime corresponds to a rational point, \Cref{def:local_dual_space} agrees with the classical dual space: indeed, when $P = \mm_p$, the $\KK$-vector spaces $D_P[I]$ and $I^\perp \cap \KK[\partial_p]$ are naturally isomorphic.
\end{remark}

We will show (\Cref{thm:zero_dim_noeth_ops}) that one can obtain Noetherian operators for a primary ideal $I \subseteq R$ by computing a $\kappa(P)$-basis of $D_{\sqrt I}[I]$. 
A natural question that arises is, for a non-primary ideal, whether one can compute Noetherian operators of a primary component without requiring generators of the primary component.
This is indeed the case for components primary to a zero-dimensional isolated prime:

\begin{proposition}\label{prop:noethOpsOfNonPrimary}
	Let $I \subseteq R$ be an ideal, let $P \subseteq R$ be a maximal ideal which is also a minimal prime of $I$, and let $Q$ be the $P$-primary component of $I$. 
	Then $D_P[I] = D_P[Q]$.
\end{proposition}
\begin{proof}
    Write $I = Q \cap I'$, where $P$ is not an associated prime of $I'$. 
    Then $D_P[I] = D_P[Q] + D_P[I']$, by \Cref{prop:local_dual_properties}(3). 
    Since $1 \not \in D_P[I']$ and $D_P[I']$ is closed under the right $R$-action, we must have $D_P[I'] = 0$.
\end{proof}

\begin{remark}\label{rmk:isol}
\Cref{prop:noethOpsOfNonPrimary} relies on $P$ being an isolated prime of $I$. 
If $P$ is an embedded prime of $I$, then a $P$-primary component of $I$ is never unique.
Moreover, there is at least one other associated prime $P'$ of $I$ strictly contained in $P$, and on writing $I = Q \cap I'$ where $P$ is not an associated prime of $I'$, it is no longer the case that $1 \notin D_P[I']$, as some associated prime of $I'$ is strictly contained in $P$. 
In fact, since $I'$ has positive dimension, $D_P[I']$ will be infinite-dimensional as a $\KK$-vector space, so there does not exist a finite basis of $D_P[I]$ to represent the multiplicity structure of $Q$.
\end{remark}

\subsection{Zero-dimensional primary ideals} \label{sec:zero_dim_case}
Throughout this subsection, $I$ denotes a zero-dimensional primary ideal in $R = \KK[\xx]$, with $P \coloneqq \sqrt I$. 

\subsubsection{Primary ideals over a rational point}
The simplest case is when $P = \mm_p$ for some $p \in \KK^n$.
The duality for $\mm_p$-primary ideals is summarized in the following:
\begin{theorem}[{\cite[Thm 2.6]{mmm93}}]\label{thm:mmm_bijection_rational_point}
	There is a bijection between $\mm_p$-primary ideals $I \subseteq R$ and finite dimensional subspaces $\DD \subseteq \KK[\partial_p]$ closed under the right $R$-action. 
	The correspondence is given by $I \mapsto I^\perp$ and $\DD \mapsto \DD^\perp$. Moreover $\dim_\KK(I^\perp) = \deg(I) = \dim_\KK(R/I)$ and $\deg(\DD^\perp) = \dim_\KK(\DD)$.
\end{theorem}

We describe how to obtain a set of Noetherian operators for an $\mm_p$-primary ideal $I$. 
First, compute a dual basis $D_1, \dotsc, D_m$ of $I^\perp$, where $D_i \in \KK[\partial_{p,1}, \dotsc, \partial_{p,n}]$.
Let $N_i \in W_R$ be the Weyl algebra element obtained by replacing $\partial_{p,i}$ with $\partial_i$. 
Then $\{N_1,\dotsc,N_m\}$ is a set of Noetherian operators for $I$: if $f \in I$, then $0 = D_i(f) = (N_i \bullet f)(p)$, which implies $N_i \bullet f \in \mm_p$ for all $i$.
Conversely, if $N_i \bullet f \in \mm_p$ for all $i$, then $D_i(f) = 0$ for all $i$, hence $f \in I^{\perp \perp} = I$.

\subsubsection{Primary ideals over a non-rational point}
Next, assume $P \ne \mm_p$ for any $p \in \KK^n$, i.e. $P$ does not correspond to any $\KK$-rational point (this happens only when $\KK$ is not algebraically closed). 
If $\cl \KK$ is the algebraic closure of $\KK$, then the extension $P_{\cl \KK}$ of $P$ to $\cl \KK[x_1,\dotsc,x_n]$ is still zero-dimensional and radical, but is no longer prime. 

\begin{proposition}\label{prop:L_is_kappa_P}
	Let $P \subseteq R$ be a maximal ideal, and $\ell \in R$ a linear form such that $\ell(p) \neq \ell(q)$ for all $p \ne q \in \VV(P_{\cl \KK})$. 
	Then:
	\begin{enumerate}
	 	\item There exist univariate polynomials $g,g_1,\dotsc,g_n$ over $\KK$ such that
	 	\begin{align*}
	 		P = ( g(\ell), x_1 - g_1(\ell), \dotsc, x_n - g_n(\ell)\}.
	 	\end{align*}
	 	Furthermore, $\deg(g_i) < \deg(g) = \deg(P)$.  
	 	\item For $p \in \VV(P_{\cl \KK})$, the field $\kappa(P) = R/P$ is the smallest extension of $\KK$ containing all coordinates of $p$.
	 \end{enumerate} 
\end{proposition}
\begin{proof}
	\begin{enumerate}
		\item This follows from the Shape lemma \cite[Proposition 1.6]{gianni1987algebraic}.
		\item It follows from (1) that
		\begin{align*}
			\kappa(P) = \frac{\KK[x_1,\dotsc,x_n]}{(g(\ell), x_1 - g_1(\ell), \dotsc, x_n - g_n(\ell))} \cong \frac{\KK[\ell]}{(g(\ell))} =: \KK(\beta),
		\end{align*}
		where $\beta$ is a solution to $g(\ell) = 0$. 
		Thus $\VV(P_{\cl\KK})$ contains the point $(g_1(\beta), \dotsc, g_n(\beta)) \in \cl\KK^n$. On the other hand, the maximal ideal in $\cl \KK[\xx]$ associated to $p$ contains a linear factor of $g(\ell)$, so any subfield of $\cl \KK / \KK$ containing all coordinates of $p$ contains an isomorphic copy of $\kappa(P)$. \qedhere
	\end{enumerate}
\end{proof}

Let $p \in \VV(P_{\cl \KK})$ be as in \Cref{prop:L_is_kappa_P}.2, and let $\mm_p \subseteq \kappa(P)[\xx]$ be the associated maximal ideal, which is now rational over the larger field $\kappa(P)$.
Then $\kappa(P) = \kappa(\mm_p)$, so $W_{\kappa(P)} = W_{\kappa(\mm_p)}$. 
This allows us to compare the $\kappa(P)$-vector spaces given by local dual spaces at $P$ and $\mm_p$:
\begin{align*}
	D_P[I] &:= \{D \in W_{\kappa(P)} \colon \langle D, f \rangle_{\kappa(P)} = 0 \text{ for all } f \in I\}\\
	D_{\mm_p}[J] &:= \{D \in W_{\kappa(P)} \colon \langle D, f \rangle_{\kappa(\mm_p)} = 0 \text{ for all } f \in J\}.
\end{align*}
Note that even though $\kappa(P) = \kappa(\mm_p)$, there are distinct pairings
\begin{align*}
	\langle \_ , \_ \rangle_{\kappa(P)} &\colon \kappa(P)[\partial] \times \KK[\xx] \to \kappa(P) \\
	\langle \_ , \_ \rangle_{\kappa(\mm_p)} &\colon \kappa(\mm_p)[\partial] \times \kappa(P)[\xx] \to \kappa(\mm_p)
\end{align*}
arising over different base rings $\KK[\xx]$ and $\kappa(P)[\xx]$. 
However, they do agree on the restriction of $\kappa(P)[\xx]$ to $\KK[\xx]$, in the sense that $\langle D, f \rangle_{\kappa(P)} = \langle D, f \rangle_{\kappa(\mm_p)}$ for $f \in \KK[\xx]$. 
In particular:
\begin{lemma}\label{lem:nonrational-to-rational}
	The $\kappa(P)$-vector spaces $D_{\mm_p}[I_{\kappa(P)}]$ and $D_P[I]$ are equal.
\end{lemma}
\begin{proof}
	The inclusion $D_{\mm_p}[I_{\kappa(P)}] \subseteq D_P[I]$ is clear since $I = I_{\kappa(P)} \cap \KK[\xx]$. 
	For the other inclusion, let $D \in D_P[I]$. 
	Let $\kappa(P)$ be generated by $\{1,k_2,\dotsc,k_e\}$ over $\KK$. 
	If $I = (f_1,\dotsc,f_n)$, then $I_{\kappa(P)}$ is also generated by $f_1,\dotsc,f_r$ in $\kappa(P)[\xx]$. 
	Hence any element $g \in I_{\kappa(P)}$ is of the form $g = \sum_i g_i f_i$ for some $g_i \in \kappa(P)[\xx]$, where the $g_i$ themselves are of the form $g_i = \sum_j g_{i,j} k_j$ for some $g_{i,j} \in \KK[\xx]$. 
	Then as $g_{i,j}f_i \in I \subset \KK[\xx]$,
	\begin{align*}
		\langle D, g \rangle_{\mm_p} = \sum_i \langle D, g_i f_i \rangle_{\kappa(\mm_p)} = \sum_i \sum_j k_j \langle D, g_{i,j}f_i \rangle_{\kappa(\mm_p)} = \sum_i \sum_j k_j \langle D, g_{i,j}f_i \rangle_{\kappa(P)} = 0.\qquad\qedhere
	\end{align*}
\end{proof}

Just as in \Cref{thm:mmm_bijection_rational_point}, there is also a correspondence theorem given in \cite{mmm93} for zero-dimensional primary ideals that are not primary to a rational point.

\begin{proposition}[{\cite[Prop 2.7]{mmm93}}]\label{prop:bijection_non_rational_mmm}
	Let $P \subseteq \KK[\xx]$ be a maximal ideal, and assume $P \neq \mm_q$ for any $q \in \KK^n$. 
	Let $P_{\kappa(P)}$ be the extension of $P$ in $\kappa(P)[\xx]$, and let $\mm_p$ be a minimal prime of $P_{\kappa(P)}$ (for some $p \in \kappa(P)^n$). 
	There is a bijection between $P$-primary ideals $I \subseteq \KK[\xx]$ and finite dimensional subspaces $\DD \subseteq \kappa(P)[\partial_p]$ closed under the right $\kappa(P)[\xx]$-action, given by
	\begin{align*}
		I &\mapsto D_{\mm_p}[I_{\kappa(P)}] \cong \{D \in \kappa(P)[\partial_p] \colon D(f) = 0 \text{ for all } f\in I_{\kappa(P)} \} \cong Q^\perp\\
		\DD &\mapsto \{f \in \kappa(P)[\xx] \colon D(f) = 0 \text{ for all } D \in \DD\} \cap \KK[\xx] \cong \mathcal{D}^\perp \cap \KK[\xx],
	\end{align*}
	where $Q \subseteq \kappa(P)[\xx]$ is the $\mm_p$-primary component of $I_{\kappa(P)}$.
\end{proposition}

\Cref{prop:bijection_non_rational_mmm} links our \Cref{def:local_dual_space} to the classical Macaulay dual spaces: the local dual space of $I$ at $P$ corresponds to a finite dimensional space of linear functionals over a field extension where $P$ contains rational point solutions. 
The connection to Noetherian operators is described in the following:

\begin{theorem}\label{thm:zero_dim_noeth_ops}
    Let $P \subseteq R$ be a maximal ideal, and $I \subseteq R$ a $P$-primary ideal.
    \begin{enumerate}
        \item  If $\{D_i\}_i$ spans $D_P[I]$ (as a $\kappa(P)$-vector space), then any preimages $\{N_i\}_i \subseteq W_R$ of $\{D_i\}_i$ is a set of Noetherian operators for $I$.
        \item Conversely, if $\{N_i\}_i \subseteq W_R$ is a set of Noetherian operators for $I$, then their images in $W_{\kappa(P)}$ span $D_P[I]$. 
    \end{enumerate}
	
\end{theorem}
\begin{proof}
    \begin{enumerate}
        \item For $f \in R$, one has $N_i \bullet f \in P \iff \langle D_i, f \rangle = 0$. If $f\in I$, then $\langle D_i, f \rangle = 0$, so $N_i \bullet f \in P$ for all $i$.
        Conversely, if $f \in R \setminus I$, then $f \not \in I_{\kappa(P)}$, so $\langle D_i, f \rangle \neq 0$ for some $i$.
        \item Let $D_i$ be the image of $N_i$ in $W_{\kappa(P)}$, and let $\mathcal{D} := \Span{D_i}$ be the $\kappa(P)$-span. 
        Then $\mathcal{D} \subseteq D_P[I]$, and if $\mathcal{D} \ne D_P[I]$, then $\mathcal{D}^\perp \supsetneq D_P[I]^\perp = I$ by \Cref{prop:bijection_non_rational_mmm}. 
        Then any $g \in \mathcal{D}^\perp \setminus I$ satisfies $N_i \bullet g \in P$ for all $i$, which is impossible since $\{N_i\}_i$ are Noetherian operators for $I$.\qedhere
    \end{enumerate}
\end{proof}

\begin{corollary} \label{cor:minimalNops}
 Let $P \subseteq R$ be a maximal ideal, and $I \subseteq R$ a $P$-primary ideal. 
 The set $\{N_i\}_i \subseteq W_R$ is a minimal set of Noetherian operators for $I$ if and only if the set of their images $\{D_i\}_i \subseteq W_{\kappa(P)}$ is a basis of $D_P[I]$.
 Conversely, the set $\{D_i\}_i \subseteq W_{\kappa(P)}$ is a basis of $D_P[I]$ if and only if any preimages $\{N_i\}_i \subseteq W_R$ is a minimal set of Noetherian operators.
\end{corollary}

\subsection{Positive dimensional primary ideals} \label{ssec:positive_dim}
Now suppose $I \subseteq R$ is a primary ideal of arbitrary dimension $d$.
Then there exists a set of $d$ variables in $R$ which is algebraically independent in $R/I$.
We refer to these as {\em independent variables} $\tt := \{t_1,\ldots,t_d\}$, the remaining variables as {\em dependent variables} $\xx := \{x_1,\ldots,x_{n-d}\}$, and write $R = \KK[\tt,\xx]$. 
Since we have two types of variables, we also write the Weyl algebra $W_R = R[\partial_{\tt},\partial_{\xx}] := R[\partial_{t_1},\dotsc,\partial_{t_d}, \partial_{x_1,\dotsc,\partial_{x_{n-d}}}]$, where $\partial_{x_i},\partial_{t_j}$ correspond respectively to $\frac{\partial}{\partial x_i}, \frac{\partial}{\partial t_j}$.
Note that after a generic linear change of coordinates, every subset of $d$ variables in $R$ is independent in $R/I$ -- this can avoid the step of computing an independent set of variables, at the cost of any structure present in the generators of $I$.
Set $U := \KK[\tt] \setminus \{0\}$, and $S := U^{-1}R = \KK(\tt)[\xx]$, the localization of $R$ at the multiplicative set $U$. 
Let $IS$ denote the extension of $I$ to $S$, which is the ideal of $S$ generated by the image of $I$ under the (injective) localization map $R \to S$. 
For any $f \in S$, there exists $u \in U$ such that $g := uf$ is in $R$ -- we call any such $g$ a \emph{lift} of $f$ in $R$, and extend this notion to the inclusion $W_R \hookrightarrow W_S$ in the natural way.
\begin{lemma} \label{lemma:zeroDimExtension}
	Let $\tt$ be a maximal independent set of variables for $I$, and $S = \KK(\tt)[\xx]$ as above. 
	Then
	\begin{enumerate}
		\item $\dim IS = 0$,
		\item $IS \cap R = I$,
		\item $\sqrt{IS} \cap R = \sqrt{I}$.
	\end{enumerate}
\end{lemma}

\begin{proof}
The algebraic independence of $\tt$ in $R/I$ means exactly that the universal map $\KK[\tt] \to R/I$ is injective. 
Since $\tt$ was maximal, none of the dependent variables $\xx$ are transcendental over $\KK[\tt]$, so localizing at $U$ gives an integral extension $\KK(\tt) \hookrightarrow U^{-1}(R/I) \cong S/IS$. 
Thus $\dim S/IS = 0$, which is (i). 
Then (ii) and (iii) follow from the fact that $I$ is primary with $I \cap U = \emptyset$ (so also $\sqrt{I} \cap U = \emptyset$), together with the 1-1 correspondence of primary (resp. prime) ideals in a localization, see \cite[Proposition 4.8(ii)]{atiyah1969introduction}.
\end{proof}

Since $IS$ is zero-dimensional, we can compute a $\kappa(PS)$-basis of $D_{PS}[IS]$ as in \Cref{sec:zero_dim_case}, and recover a set of Noetherian operators for $I$ from this basis:

\begin{proposition}\label{prop:pos_dim_noeth_ops}
	Let $I \subseteq R$ be a primary ideal of dimension $d$, $P = \sqrt I$, and $S = \KK(\tt)[\xx]$ where $\tt, \xx$ are independent resp. dependent variables for $I$.
	\begin{enumerate}
		\item If $\{D_i\}_i \subseteq W_S$ is a set of Noetherian operators for $IS$, then any lift $\{N_i\}_i \subseteq W_R$ of $\{D_i\}_i$ is a set of Noetherian operators for $I$.
		\item Conversely, if $\{N_i\}_i \subseteq W_R$ is a set of Noetherian operators of $I$ whose differential variables involve only $\partial_{\xx}$ (and not $\partial_{\tt}$), then their images $\{D_i\}_i \subseteq W_S$ is a set of Noetherian operators for $IS$.
	\end{enumerate}
\end{proposition}
\begin{proof}
\begin{enumerate}
	\item For $f \in S$, one has $N_i \bullet f \in R$ for all $i \iff f \in R$: $\Leftarrow$ follows since the $N_i$ have coefficients in $R$, and $\Rightarrow$ follows since $1 \in D_{PS}[IS]$ is in the span of $\{D_i\}_i$. 
	Then by \Cref{lemma:zeroDimExtension}, $f \in I = IS \cap R \iff N_i \bullet f \in PS \cap R$ for all $i \iff N_i \bullet f \in P$ for all $i$.

	\item We show that $f \in IS \iff N_i \bullet f \in PS$ for all $i$. 
	For the forward direction, let $f \in IS$. 
	Then $f = \frac{g}{u}$ for some $g \in I$, $u \in U$.
	For every $i$, we have that $\frac{1}{s}$ is a scalar with respect to $D_i$ (since $N_i$ involves only $\partial_{\xx}$), so $D_i \bullet f = \frac{N_i \bullet g}{u} \in PS$, since $N_i \bullet g \in P$.

	Conversely, suppose $f = \frac{g}{u} \in S$ ($g \in R$, $u \in U$) is such that $N_i \bullet f \in PS$ for all $i$. 
	Then $g = uf$ in $R$ (since $R \hookrightarrow S$ is injective), so $N_i \bullet g = N_i \bullet (uf) = u(N_i\bullet f) \in PS \cap R = P$ for all $i$. 
	Hence $g \in I$, and thus $f \in IS$. \qedhere
\end{enumerate}
\end{proof}

\subsection{Symbolic Algorithms} \label{sec:algs_and_exs}
In this subsection, we present algorithms to symbolically compute bases for local dual spaces, which yields Noetherian operators by \Cref{thm:zero_dim_noeth_ops} and \Cref{prop:pos_dim_noeth_ops}. 
The method is a straightforward adaptation of the classical theory of Macaulay inverse systems involving \emph{Macaulay matrices}.

As usual, we start with the zero-dimensional case. 
Let $I$ be a zero-dimensional ideal in $R = \KK[\xx]$ and $P$ a minimal prime of $I$. 
Given $D \in W_{\kappa(P)}$, we say the \emph{$\partial$-degree} of $D$ is $d$ if $D$ is a degree $d$ polynomial in the $\partial$-variables with coefficients in $\kappa(P)$. 
We define the \emph{degree $d$ truncated local dual spaces} as
\begin{align*}
	D_P^{(d)}[I]:= \{D \in D_P[I] \mid \text{$\partial$-degree of $D$ is $\le d$}\}.
\end{align*}

As in \Cref{sec:dual-spaces-for-nonrational-points} let $\kappa(P)$ be the residue field of $P$, $p \in \VV(P_{\kappa(P)})$ and $\mm_p$ the maximal ideal at $p$. 
By \Cref{lem:nonrational-to-rational} and \Cref{prop:bijection_non_rational_mmm}, 
\[ D_P[I] = D_{\mm_p}[I_{\kappa(P)}] = \{D \in \kappa(P)[\partial_p]: D(f) = 0 \text{ for all } f \in I_{\kappa(P)}\}. \]
Both $\{\partial^\alpha\}_{\alpha \in \ZZp^n}$ and $\{\partial_p^\alpha\}_{\alpha \in \ZZp^n}$ are bases for $\kappa(P)[\partial]$, and for $D \in \kappa(P)[\partial]$ the $\partial$-degree and $\partial_p$-degree of $D$ are equal.

Fix a $\partial$-degree $d$, and let $C := \{\partial^\beta \mid |\beta| \leq d\}$, the set of all $\partial$-monomials of $\partial$-degree at most $d$. 
Pick a generating set $\{f_1,\dotsc,f_r\}$ for $I$, and let $F := \{\xx^\alpha f_i \mid i = 1,\dotsc, r,\; |\alpha| < d\}$. 
For a fixed total ordering $\prec$ on $\partial$-monomials, we define the \emph{degree $d$ Macaulay matrix} $M$ of dimension $|F| \times |C|$, where the rows are indexed by $F$, and the columns are indexed by $C$ and ordered according to $\prec$. 
The entry corresponding to the row $\xx^\alpha f_i$ and column $\partial^\beta$ of the Macaulay matrix is the value with respect to the pairing \Cref{def:pairing}, i.e.
\begin{align*}
	M_{\alpha,i;\beta} = \langle \partial^\beta, \xx^\alpha f_i \rangle_{\kappa(P)} \in \kappa(P).
\end{align*}
Any $D = \sum_{|\beta|\leq d} v_\beta \partial^\beta \in W_{\kappa(P)}$ is specified by its coefficient (column) vector $v = (v_\beta)_\beta$. 
Every entry of $Mv$ is of the form $\langle D, g \rangle$ for some $g \in I$, so every element in the truncated local dual space $D_P^{(d)}[I]$ corresponds to a vector in the kernel of the Macaulay matrix. 
To show the reverse, we need the following:

\begin{lemma}\label{lem:truncated_dual_power}
  With notation as above,
  \begin{align*}
    D_P^{(d)}[I] = D_P[I + P^{d+1}].
  \end{align*}
\end{lemma}
\begin{proof}
First we show that
\[ D_P[P^{d+1}] = D^{(d)}_P[0] = \spn_{\kappa(P)}\{\partial_p^\beta : |\beta| \leq d\}. \]
Indeed, since $P_{\kappa(P)}$ is a product of maximal ideals, localizing $P_{\kappa(P)}$ at $\mm_p$ gives $D_P[P^{d+1}] = D_{\mm_p}[\mm_p^{d+1}]$, by \Cref{lem:nonrational-to-rational}. 
If $|\alpha| > d$ and $|\beta| \leq d$, then $\langle \partial_p^\beta, (\xx-p)^\alpha\rangle = 0$, so $D^{(d)}_P[0] \subseteq D_{\mm_p}[\mm_p^{d+1}]$ (as $\mm_p^{d+1}$ is spanned over $\kappa(P)$ by $\{(\xx-p)^\alpha \mid |\alpha| > d\}$).
Conversely, if $D$ has $\partial_p$-degree $> d$ then it has a nonzero term $c_\alpha \partial_p^\alpha$ with $|\alpha| > d$. 
Then $\langle D, (\xx-p)^\alpha \rangle = \alpha ! c_\alpha \neq 0$, hence $D \notin D_{\mm_p}[\mm_p^{d+1}]$.
  
Applying \Cref{prop:local_dual_properties}(2) then yields
\[ D_P[I + P^{d+1}] = D_P[I] \cap D_P[P^{d+1}] = D_P[I] \cap D^{(d)}_P[0] = D_P^{(d)}[I]. \qedhere \]
\end{proof}

\begin{proposition}
	With notation as above, let $\{v^{(k)}\}_k$ be a basis of the kernel of the degree $d$ Macaulay matrix, and let $D_k := \sum_\beta v^{(k)}_\beta \partial^\beta$. 
	Then $\{D_k\}_k$ is a basis for the truncated local dual space $D^{(d)}_P[I]$.
\end{proposition}
\begin{proof}
Let $D \in D^{(d)}_P[I]$. 
We can write $D = \sum_{|\beta|\leq d} v_\beta \partial^\beta$ for some vector $v = (v_\beta)_\beta$.
Clearly $v \in \ker M$, so $v = \sum_k c_k v^{(k)}$, which implies $D = \sum_k c_k D_k$.
	
Conversely, we must show that $D_k$ is in $D^{(d)}_P[I]$ for each $k$. 
The set 
\[ \{\xx^\alpha f_i \mid |\alpha| < d,\; i = 1,\ldots,r\} \cup \{(\xx-p)^\beta f_i \mid |\beta| \geq d, \; i = 1,\ldots,r\}\]
spans $I_{\kappa(P)}$. 
By construction, $\langle D_k, \xx^\alpha f_i \rangle = 0$ for all $|\alpha| < d$. 
Note that each $f_i$ vanishes at $p$, so $f_i \in \mm_p$. 
For each $j$, the term $x_j - p_j$ is also in $\mm_p$. 
If $|\beta| \geq d$ then $(\xx-p)^\beta f_i \in \mm_p^{d+1}$.
Since the $\partial$-degree of $D_k$ is at most $d$, $D_k \in D_{\mm_p}[\mm_p^{d+1}]$ by \Cref{lem:truncated_dual_power}. 
So $\langle D_k, (\xx-p)^\beta f_i \rangle = 0$. 
Therefore $D_k \in D_{\mm_p}[\mm_p^{d+1}] \cap D_P[I] = D^{(d)}_P[I]$.
\end{proof}
It is clear that $D_P^{(1)}[I] \subseteq D_P^{(2)}[I] \subseteq \dotsb$, and since the local dual space is finite dimensional, this chain will stabilize to $D_P[I]$ after a finite number of steps. 
Furthermore, as the $D_p^{(d)}[I]$ are closed under the right $R$-action, the chain stabilizes when $\dim_{\kappa(P)} D_P^{(d)}[I] = \dim_{\kappa(P)} D_P^{(d+1)}[I]$.
In view of \Cref{prop:noethOpsOfNonPrimary}, we thus arrive at \Cref{alg:symb_zero_dim}, which computes Noetherian operators for the $P$-primary component of $I$ via kernels of successively larger Macaulay matrices. 
The algorithm computes the local dual space, and then constructs Noetherian operators from a basis thereof, so the output Noetherian operators will depend on a choice of basis of the local dual space. 
In our Macaulay2 implementation, we always choose a basis in reduced column echelon form.

\begin{algorithm}
\caption{Compute Noetherian operators symbolically in dimension zero}
\begin{algorithmic}[1]
\Require $I = \ideal{f_1,\dotsc,f_r}$ a zero-dimensional ideal, $P$ a minimal prime of $I$, $\prec$ an ordering on monomials $\partial^\beta$
\Ensure A set of Noetherian operators for the $P$-primary component of $I$
\Procedure{NoetherianOperatorsZero}{$I, P$}
    \State $K \gets \emptyset$
    \State $d \gets 0$ \Comment{$d$ corresponds to the degree bound}
    \Repeat
        \State $d \gets d+1$
        \State $F \gets $ vector with entries $\xx^\alpha f_i$, where $|\alpha| < d$, $i = 1,2,\dotsc,r$
        \State $C \gets $ vector with entries $\partial^\beta = \partial_{x_1}^{\beta_1}\dotsb\partial_{x_{n}}^{\beta_{n}}$, where $|\beta| \leq d$, in the order given by $\prec$
        \State $M \gets $ the Macaulay matrix with entries $\langle \partial^\beta, \xx^\alpha f_i \rangle_{\kappa(P)}$ (rows indexed by $F$, columns by $C$)
        \State $K_b \gets \ker M$
    \Until{$\dim K_b = \dim K_{b-1}$} \Comment{Stop when the dimension of the kernel stabilizes}
    \State $K \gets $ \Call{ColReduce}{$K_b$} \Comment{Rewrites the generators of $K_b$ in a reduced column echelon form}\label{line:symb_col_reduce}
    \State \textbf{return } preimage of $C^TK$ in $W_R$
\EndProcedure
\end{algorithmic}
\label{alg:symb_zero_dim}
\end{algorithm}

For the general case, if $I$ is positive-dimensional, we can use \Cref{prop:pos_dim_noeth_ops} to reduce to the zero dimensional case, yielding \Cref{alg:symb_pos_dim}.

\begin{algorithm}
\caption{Compute Noetherian operators symbolically in positive dimension}
\begin{algorithmic}[1]
\Require $I \subseteq \KK[\tt,\xx]$ an ideal, where $\tt,\xx$ are independent and dependent variables for $I$ respectively, $P$ a minimal prime of $I$, $\prec$ an ordering on monomials $\partial_{\xx}^\beta$
\Ensure A set of Noetherian operators for the $P$-primary component of $I$
\Procedure{NoetherianOperators}{$I, P$}
	\State $S \gets \KK(\tt)[\xx]$
    \State $K \gets \Call{NoetherianOperatorsZero}{IS, PS}$
    \State \textbf{return } lift of $K$ in $W_R$
\EndProcedure
\end{algorithmic}
\label{alg:symb_pos_dim}
\end{algorithm}

\begin{example}
Consider the 1-dimensional primary ideal $Q = ( (x_1^2 - x_3)^2, x_2 - x_3(x_1^2 - x_3) ) \subseteq R = \QQ[x_1,x_2,x_3]$. 
Its radical is $P = (x_1^2 - x_3, x_2)$, and we may choose $x_1,x_2$ as the dependent variables and $x_3$ as the independent variable. 
Thus in $S = \QQ(x_3)[x_1,x_2]$, $QS$ is a zero-dimensional primary ideal whose radical is $PS$. 
In degree 1, the Macaulay matrix has a 2-dimensional kernel.
In degree 2, the Macaulay matrix is

\begin{align*}
M = \kbordermatrix{
& 1 & \partial_{x_1} & \partial_{x_2} & \partial_{x_1}^2 & \partial_{x_1}\partial_{x_2} & \partial_{x_2}^2\\
(x_1^2-x_3)^2 & 0 & 0 & 0 & 8\,{x}_{3} & 0 & 0\\
(x_2 - x_3(x_1^2-x_3)) & 0 & {-2\,{x}_{3}{x}_{1}} & 1 & {-2\,{x}_{3}} & 0 & 0\\
x_1(x_1^2-x_3)^2 & 0 & 0 & 0 & 8\,{x}_{3}{x}_{1} & 0 & 0\\
x_1(x_2 - x_3(x_1^2-x_3)) & 0 & {-2\,{x}_{3}^{2}} & {x}_{1} & {-6\,{x}_{3}{x}_{1}} & 1 & 0\\
x_2(x_1^2-x_3)^2 & 0 & 0 & 0 & 0 & 0 & 0\\
x_2(x_2 - x_3(x_1^2-x_3)) & 0 & 0 & 0 & 0 & {-2\,{x}_{3}{x}_{1}}&2
}
\end{align*}
with entries in $S/PS$. 
Performing linear algebra in the field $S/PS$, we see that the kernel of $M$ is generated by $(1,0,0,0,0,0)^T$ and $(0,1,2x_1x_3,0,0,0)^T$. 
Since the dimension of the kernel did not increase, we terminate the loop in \Cref{alg:symb_zero_dim} and conclude that $\{1, \partial_{x_1} + 2x_1x_3 \partial_{x_2}\}$ is a set of Noetherian operators for $Q$.

Contrary to the algorithm in \cite{yairon_roser_bernd}, our algorithm does not go through the \emph{punctual Hilbert scheme}. 
To make this clear, we perform a parallel computation following \cite[Alg. 3.8]{yairon_roser_bernd}. 
Let $\FF$ denote the field of fractions of the integral domain $R/P$. 
The point in the punctual Hilbert scheme corresponding to $Q$ is the ideal
\begin{align*}
I = \langle y_1, y_2 \rangle^2 + \gamma(Q) \cdot \FF[y_1,y_2],
\end{align*}
where $\gamma$ is the inclusion map
\begin{align*}
\gamma \colon R \hookrightarrow \FF[y_1,y_2], ~~ \begin{array}{l}
x_1 \mapsto y_1 + x_1\\
x_2 \mapsto y_2 + x_2\\
x_3 \mapsto x_3\\
\end{array}
\end{align*}
Here $I = (y_1 - 1/(2x_1x_3)y_2, y_2^2)$. A basis for $I^\perp$ can be computed using e.g. the classical Macaulay matrix method. The degree 2 Macaulay matrix is
\begin{align*}
\kbordermatrix{
& 1 & \partial_{x_1} & \partial_{x_2} & \partial_{x_1}^2 & \partial_{x_1}\partial_{x_2} & \partial_{x_2}^2\\
(y_1 - 1/(2x_1x_3)y_2) & 0 & 1 & \frac{{-1}}{2\,{x}_{1}{x}_{3}} & 0 & 0 & 0\\
y_2^2 & 0 & 0 & 0 & 0 & 0 & 2\\
y_1(y_1 - 1/(2x_1x_3)y_2) & 0 & 0 & 0 & 2 & \frac{{-1}}{2\,{x}_{1}{x}_{3}} & 0\\
y_1y_2^2 & 0 & 0 & 0 & 0 & 0 & 0\\
y_2(y_1 - 1/(2x_1x_3)y_2) & 0 & 0 & 0 & 0 & 1 & \frac{{-1}}{{x}_{1}{x}_{3}}\\
y_2y_2^2 & 0 & 0 & 0 & 0 & 0 & 0
},
\end{align*}
with entries in $\FF$, and, as expected, its kernel corresponds to the Noetherian operators $\{1, \partial_{x_1} + 2x_1x_3 \partial_{x_2}\}$.
\end{example}

\section{A numerical approach via interpolation} \label{sec:numerical_approach}
Keeping notation from \Cref{ssec:positive_dim}, let $I \subseteq \KK[\tt,\xx]$ be a primary ideal of dimension $d$, where $\tt$ and $\xx$ are sets of independent and dependent variables for $I$ respectively. 
Let $\{N_1, \ldots, N_m\}$ be a set of Noetherian operators for $I$ as in \Cref{prop:pos_dim_noeth_ops}, and write
\begin{align*}
	N_i := \sum_\alpha f_{\alpha,i}(\tt,\xx) \partial_{\xx}^{\alpha}.
\end{align*}
Fix a point $(\tt_0,\xx_0) \in \VV(I)$ on the variety of $I$. We denote by $N_i(\tt_0,\xx_0)$ the \emph{specialized Noetherian operator}
\begin{align*}
	N_i(\tt_0,\xx_0) = \sum_\alpha f_{\alpha,i}(\tt_0,\xx_0) \partial_{\xx}^{\alpha} \in \KK[\partial_{\xx}].
\end{align*}
\begin{theorem} \label{thm:numerical_noetherian_operators}
	Assume $\KK = \cl \KK$. Let $\{N_1,\dotsc,N_m\}$ be a minimal set of Noetherian operators of a primary ideal $I$, and let $(\xx_0,\tt_0) \in \VV(I)$. If $\tt_0$ is general, then
	\begin{align*}
		\spn_\KK \{N_1(\tt_0,\xx_0), \dotsc, N_m(\tt_0,\xx_0)\} = D_{\mm_{(\tt_0,\xx_0)}}[I + (\tt-\tt_0)].
	\end{align*}
\end{theorem}
\begin{proof}
	We first show that $D_{\mm_{(\tt_0,\xx_0)}}[(\tt-\tt_0)] = \KK[\partial_{\xx}]$. 
	The inclusion $\supseteq$ is clear. 
	For the opposite inclusion, we first note that every element $D \in D_{\mm_{(\tt_0,\xx_0)}}[(\tt-\tt_0)]$ can be written in the form
	\begin{align*}
		D = \sum_{\alpha,\beta} c_{\alpha,\beta} \partial_{\xx}^\alpha \partial_{\tt}^\beta,
	\end{align*}
	where $c_{\alpha,\beta} \in \KK$, $\alpha \in \NN^{n-d}$, $\beta \in \NN^d$, and only finitely many of the $c_{\alpha,\beta}$ are nonzero.
	We need to show that for all $\beta$ such that $\beta_1+\dotsb+\beta_d > 0$ we have $c_{\alpha,\beta} = 0$. 
	Assume this is not the case. 
	Since the local dual space is closed under the right R-action, we can repeatedly act on $D$ from the right with elements of the form $(t_i - (\tt_0)_i)$ and $(x_i - (\xx_0)_i)$ to obtain an operator $D' \in D_{\mm_{(\tt_0,\xx_0)}}[(\tt-\tt_0)]$ that has degree 1 in $\partial_{\tt}$-variables and degree 0 in $\partial_{\xx}$-variables. 
	More precisely, we get an operator
	\begin{align*}
		D' = c_0 + \sum_{i=1}^d c_i \partial_{t_i},
	\end{align*}
	where $c_j \in \KK$, $j = 0,\dotsc,d$, and $c_i \neq 0$ for at least one $i = 1,\dotsc,d$. 
	In this case however, we have
	\begin{align*}
		\langle D', (t_i - (\tt_0)_i) \rangle = c_i \neq 0,
	\end{align*}
	which is a contradiction.

	With this, \Cref{prop:local_dual_properties} yields that \[D_{\mm_{(\tt_0,\xx_0)}}[I + (\tt-\tt_0)] = D_{\mm_{(\tt_0,\xx_0)}}[I] \cap D_{\mm_{(\tt_0,\xx_0)}} [(\tt-\tt_0)] = D_{\mm_{(\tt_0,\xx_0)}}[I] \cap \KK[\partial_{\xx}].\]
	
    Since $\tt_0 \in \KK^d$ is general, the specializations $\{N_1(\tt_0,\xx_0), \dotsc, N_m(\tt_0,\xx_0)\}$ are $\KK$-linearly independent in $D_{\mm_{(\tt_0,\xx_0)}}[I + (\tt-\tt_0)]$. 
    Thus, to prove the theorem, it suffices to show that $\dim_\KK D_{\mm_{(\tt_0,\xx_0)}}[I + (\tt-\tt_0)] = m$, where $m = m(I,P)$ is the multiplicity of $I$ over $P$.
	
	Set $R_0 := R_{\mm_{(\tt_0,\xx_0)}}$, the localization of $R$ at the maximal ideal $\mm_{(\tt_0,\xx_0)}$, $I_0 := IR_0$, $P_0 := PR_0$, and $$J_0 := (I+(\tt-\tt_0))R_0 = I_0 + (\tt-\tt_0)R_0,$$ which is primary to the maximal ideal in $R_0$. 
	Then $$\dim D_{\mm_{(\tt_0,\xx_0)}}[I+(\tt-\tt_0)] = \dim D_{\mm_{(\tt_0,\xx_0)}}[J_0] = \dim_\KK R_0/J_0$$ by \Cref{thm:mmm_bijection_rational_point}. 
	On the other hand, $(\tt - \tt_0)R_0$ is a parameter ideal for $R_0/I_0$ and $R_0/P_0$. 
	By generality of $\tt_0$ again, Bertini's Theorem gives that $\tt - \tt_0$ forms a regular sequence on $R_0/I_0$, and $P_0 + (\tt-\tt_0)$ is radical, which implies $P_0+(\tt-\tt_0) = m_{(\tt_0,\xx_0)}$. 
	Thus, for general $\tt_0$, \cite[Exercise 12.11(d),(e)]{eisenbud2013commutative} implies that $$m(I,P) = m(I_0,P_0) = \dfrac{e((\tt-\tt_0), R_0/I_0)}{e((\tt-\tt_0), R_0/P_0)} = e((\tt-\tt_0), R_0/I_0) = \dim_\KK R_0/J_0$$ as desired (here $e(\mathfrak{q}, M)$ is the Hilbert-Samuel multiplicity of the parameter ideal $\mathfrak{q}$ on a module $M$). 
\end{proof}

Using the above result we obtain a numerical algorithm that computes Noetherian operators specialized at points, described in \Cref{alg:num_noeth_nops_at_point}. 
This algorithm is very similar to the symbolic algorithm for computing Noetherian operators, the only difference being that the Macaulay matrix is evaluated at a point.
The column reduction in step \ref{alg:numNoethOpsColReduce} is used to construct a basis consistent with the one computed in the symbolic algorithm. 
More precisely, for a fixed ordering $\prec$, the numerical matrix $K(p)$ in \Cref{alg:num_noeth_nops_at_point} is precisely the symbolic matrix $K$ in \Cref{alg:symb_zero_dim} evaluated at the point $p$. 
Thus if the output of \textsc{NoetherianOperators}$(I, P)$ is $\{N_1(\tt,\xx),\dotsc,N_m(\tt,\xx)\}$, then the output of \textsc{NoetherianOperatorsAtPoint}$(I, (\tt_0,\xx_0))$ will be $\{N_1(\tt_0,\xx_0),\dotsc,N_m(\tt_0,\xx_0)\}$. 
In general, \Cref{alg:num_noeth_nops_at_point} will be faster than \Cref{alg:symb_pos_dim}, as computations in the former are done in the base field $\kappa(\mm_{(\tt_0,\xx_0)}) = \KK$ rather than in $\kappa(P)$, which is an extension of the rational function field in $\tt$.

\begin{algorithm}
\caption{Compute specializations of Noetherian operators at a point}
\begin{algorithmic}[1]
\Require $I \subseteq \KK[\tt,\xx]$ an ideal, where $\tt,\xx$ are independent and dependent variables for $I$ respectively, $P$ a minimal prime of $I$, $\prec$ an ordering on monomials $\partial_{\xx}^\gamma$, and $p \in \B V(P)$
\Ensure A set of Noetherian operators for the $P$-primary component of $I$, specialized at $p$
\Procedure{NoetherianOperatorsAtPoint}{$I, p$}
    \State $K \gets \emptyset$
    \State $d \gets 0$ \Comment{$d$ corresponds to the degree bound}
    \Repeat
        \State $d \gets d+1$
        \State $F \gets $ vector with entries $\xx^\alpha \tt^\beta f_i$, where $|\alpha + \beta| < d$, $i = 1,2,\dotsc,r$
        \State $C \gets $ vector with entries $\partial^\gamma_{\xx}$, where $|\gamma| \leq d$, in the order given by $\prec$
        \State $M \gets $ the Macaulay matrix with entries $(\partial^\gamma_{\xx} \bullet (\xx^\alpha \tt^\beta f_i))(p)$ (rows indexed by $F$, columns by $C$)
        \State $K_b \gets \ker M$
    \Until{$\dim K_b = \dim K_{b-1}$} \Comment{Stop when the dimension of the kernel stabilizes}
    \State $K(p) \gets $ \Call{ColReduce}{$K_b$} \Comment{Rewrites generators of $K_b$ in reduced column echelon form} \label{alg:numNoethOpsColReduce}
    \State \textbf{return } $C^TK(p)$
\EndProcedure
\end{algorithmic}
\label{alg:num_noeth_nops_at_point}
\end{algorithm}

\begin{example}\label{ex:interpolation}
Let $I = \ideal{ x^2 - ty, y^2 }$ be an ideal in $\B C[t,x,y]$. 
Here $t$ is an independent variable, and $x,y$ are dependent.
We sample four points $(1,0,0), (2,0,0), (3,0,0), (4,0,0)$ on the variety $\B V(I)$. 
Running \Cref{alg:num_noeth_nops_at_point} gives four differential operators with constant coefficients for each point, shown in \Cref{tbl:interpolationExample}. 
	
\begin{table}[ht]
\caption{Specialized Noetherian operators at different points}
\centering
\begin{tabular}{ccccc}
\toprule
$(t,x,y)$ &  Operator 1 & Operator 2 & Operator 3 & Operator 4\\
\midrule
$(1,0,0)$ & $1$ & $\partial_x$ & $\partial_x^2 + 2 \partial_y$ & $\partial_x^3 + 6 \partial_x \partial_y$\\
$(2,0,0)$ & $1$ & $\partial_x$ & $\partial_x^2 + \partial_y$ & $\partial_x^3 + 3 \partial_x \partial_y$\\
$(3,0,0)$ & $1$ & $\partial_x$ & $\partial_x^2 + \frac{2}{3} \partial_y$ & $\partial_x^3 + 2 \partial_x \partial_y$\\
$(4,0,0)$ & $1$ & $\partial_x$ & $\partial_x^2 + \frac{1}{2} \partial_y$ & $\partial_x^3 + \frac{3}{2} \partial_x \partial_y$\\
\bottomrule
\end{tabular}
\label{tbl:interpolationExample}
\end{table}
	
Interpolating each coefficient, we conclude that the coefficient of $\partial_y$ in the third operator can be chosen to be $\frac{2}{t}$, and the coefficient of $\partial_x\partial_y$ in the fourth one can be chosen to be $\frac{6}{t}$. 
Hence we get a set of four Noetherian operators
\begin{align*}
1, \quad \partial_x, \quad \partial_x^2 + \frac{2}{t}\partial_y, \quad \partial_x^3 + \frac{6}{t}\partial_x \partial_y,
\end{align*}
which can be confirmed by symbolically computing Noetherian operators using \Cref{alg:symb_pos_dim}.
\end{example}

\subsection{Reconstructing a \setOfNOs{} from sampled points}

Given an ideal $I$ and an oracle for sampling points on an isolated component $V$ of $\B V(I)$, we seek to produce a set of Noetherian operators describing the primary ideal $Q$ corresponding to $V$. 
One way to supply such an oracle is via numerical irreducible decomposition \cite{sommese2001numerical} to construct a \emph{witness set} for each isolated component. 
The witness set for $V$ can then be used to sample points on $V$, as described in \cite{Sommese-Wampler-Verschelde-intro}.

Another instance in which such an oracle can be obtained is when the variety $V$ of interest is expressed as the image of a known rational map from another variety $W$ for which one has a witness set, $\varphi: W \dashedRightArrow V$ (cf. \emph{pseudo-witness set} from \cite{hauenstein2010witness}).
In this case points sampled from $W$ can be mapped forward to points on $V$.
In particular when $W = \KK^m$, sampling points on $\KK^m$, and therefore on $V$, is trivial.

As in \Cref{alg:symb_pos_dim}, let $I \subseteq \KK[\tt,\xx]$ be $P$-primary, $S = \KK(\tt)[\xx]$, and $K$ a basis for the kernel of the Macaulay matrix over $S$ in \Cref{alg:symb_zero_dim}. 
The entries of $K$ are coefficients of elements in $D_{PS}[IS]$, which live in the residue field $\kappa(PS) = S/PS$, and are represented by polynomials in $\xx$ with coefficients which are rational functions in $\tt$. 
On the other hand, entries of $K(p)$ in \Cref{alg:num_noeth_nops_at_point} are evaluations of the aforementioned rational functions at a point $p$, and live in $\KK$. 
We now seek to recover $K$ from a sampled set of evaluations $K(p_1),\dotsc,K(p_\ell)$, via interpolation of rational functions.
The interpolation procedure is described as follows: we wish to find a rational function $\frac{f(\tt,\xx)}{g(\tt,\xx)}$ such that $f(p_i) / g(p_i) = c_i$ for all $i = 1,\dotsc,\ell$. 
Choose an ansatz for $f,g$ of the form $f = \sum_{(\alpha,\beta) \in A} f_{\alpha,\beta} \tt^\alpha \xx^\beta$, and $g = \sum_{(\alpha,\beta) \in B} g_{\alpha,\beta} \tt^\alpha \xx^\beta$, where $A,B \subseteq \B Z_{\geq 0}^n$, with the $f_{\alpha,\beta}, g_{\alpha,\beta}$ to be determined. 
Then for each point $p_i$ we get a linear equation
\begin{align} \label{eqn:rational_matrix}
\sum_{(\alpha,\beta) \in A} f_{\alpha,\beta} p_i^{(\alpha,\beta)} - c_i \sum_{(\alpha,\beta) \in B} g_{\alpha,\beta} p_i^{(\alpha,\beta)} = 0,
\end{align}
where $p_i^{(\alpha,\beta)}$ is the monomial $\tt^\alpha \xx^\beta$ evaluated at the point $p_i$. 
Since in \eqref{eqn:rational_matrix} we are solving $f(p_i) - c_i g(p_i) = 0$, a possible solution obtained from the algorithm may correspond to a rational function $f/g$ where both $f,g \in \sqrt I$. 
For this reason, we remove the solutions where the numerator or the denominator vanishes on a generic point in $\B V(I)$. This method is described in \Cref{alg:rationalInterpolation}.

\begin{algorithm}
  \caption{Multivariate rational function interpolation}
    \begin{algorithmic}[1]
    \Require A sequence of points $p = (p_i)$ and values $v = (v_i)$; row vectors $\vec n, \vec d$ specifying the monomials appearing in the numerator and denominator
    \Ensure A rational function $f/g$ such that $\frac{f(p_i)}{g(p_i)} = v_i$ for all $i$, and where $f$ and $g$ have  monomial support in $\vec n$ and $\vec d$ respectively
    \Procedure{RationalInterpolation}{$p,v,\vec n,\vec d$}
        \State $N \gets $ matrix, whose $i$th row is the vector $\vec n$ evaluated at $p_i$
        \State $D \gets $ matrix, whose $i$th row is the vector $-v_i \vec d$ evaluated at $p_i$
        \State $M \gets \begin{pmatrix} N & D\end{pmatrix}$
        \State $K \gets \ker(M)$
        \ForAll{columns $k$ in $K$}
            \State $k_f \gets$ first \Call{Length}{$\vec n$} entries of $k$
            \State $k_g \gets$ last \Call{Length}{$\vec d$} entries of $k$
            \State $f \gets \vec n x_f$
            \State $g \gets \vec d x_g$
            \If{ $f(p_0) = 0 $ or $g(p_0) = 0$}
                \State remove column $k$ from $K$
            \EndIf
        \EndFor
        \If{$K$ is empty}
            \State \Return{error} \Comment{No suitable rational functions found}
        \EndIf
        \State $x_f \gets$ first \Call{Length}{$\vec n$} entries of any vector in $K$
        \State $x_g \gets$ last \Call{Length}{$\vec d$} entries of any vector in $K$
        \State $f \gets \vec n x_f$
        \State $g \gets \vec d x_g$
        \State \Return $\frac{f}{g}$
    \EndProcedure
    \end{algorithmic}
    \label{alg:rationalInterpolation}
\end{algorithm}

\begin{remark}
One has freedom to choose any plausible ansatz for $f, g$.
For instance one can take all rational functions in $\tt$ and $\xx$ with degrees of numerators and denominators bounded by some constant $k$. 
Then any sufficiently large $k$ is guaranteed to capture the operators we seek. 
This is the method used in our Macaulay2 implementation.

Other types of ansatzes for coefficients of operators are possible: for instance, one can choose a generating set of monomials in $\xx$ for the residue field $\kappa(PS)$ as an extension of $\KK(\tt)$, together with a degree bound on numerators and denominators of rational functions in $\tt$.  
\end{remark}

Combining the subroutines in \Cref{alg:num_noeth_nops_at_point,alg:rationalInterpolation}, we obtain \Cref{alg:mainAlgorithm}, the main numerical algorithm for computing Noetherian operators.
The algorithm takes as input an ideal and an oracle for sampling points on $\B V(I)$, and outputs a set of Noetherian operators with interpolated rational function coefficients.

\begin{algorithm}
  \caption{Compute Noetherian operators numerically via interpolation}
    \begin{algorithmic}[1]
    \Require $I \subseteq \KK[\tt,\xx]$ an ideal, $p = (p_i)$ a sequence of points in $\B V(P)$ where $P$ is an isolated prime of $I$
    \Ensure A set of Noetherian operators for the $P$-primary component of $I$
    \Procedure{NumericalNoetherianOperators}{$I,p$}
        \ForAll{$i = 1,2,\dotsc$}
            \State $\mathcal N_i \gets $ \Call{NoetherianOperatorsAtPoint}{$I, p_i$}
        \EndFor
        \ForAll{terms $c_\alpha \partial^\alpha$ appearing in elements of $\mathcal N_1$} \Comment{$c_\alpha \in \KK$}
            \State $v_i \gets c_\alpha$ for the corresponding term $c_\alpha \partial^\alpha$ in $\mathcal N_i$ for all $i$
            \State $d \gets 0$
            \Repeat
            	\State $\vec n \gets $monomials $\xx^\alpha \tt^\beta$ such that $|\alpha+\beta| \leq d$.
            	\State $\vec d \gets$ monomials $\tt^\gamma$ such that $|\gamma| \leq d$
            	\State $f_\alpha / g_\alpha \gets $\Call{RationalInterpolation}{$p,v, \vec n,\vec d$}
            	\State $d \gets d+1$
            \Until{interpolation succeeds}
        \EndFor
        \State \Return the set of operators $\mathcal N_1$ in which each term $c_\alpha \partial^\alpha$ is replaced by $\frac{f_\alpha}{g_\alpha} \partial^\alpha$
    \EndProcedure
    \end{algorithmic}
    \label{alg:mainAlgorithm}
\end{algorithm}

Finally, combining \Cref{alg:mainAlgorithm} with an existing numerical irreducible decomposition procedure yields \Cref{alg:num_primary_decomp}, a numerical primary decomposition algorithm for unmixed ideals.

\begin{algorithm}
\caption{Numerical primary decomposition for unmixed ideals}
\begin{algorithmic}[1]
\Require $I \subseteq \KK[\tt,\xx]$ an unmixed ideal
\Ensure A list of irreducible components of $V(I)$ and a set of Noetherian operators for each primary component of $I$
\Procedure{NumericalPrimaryDecomposition}{$I$}
	\State $NV \gets \Call{NumericalIrreducibleDecomposition}{I}$
	\State output $\gets \{ \}$ 
	\For{$W$ in $NV$}
	    \State $p \gets \Call{sample}{W}$
	    \State $N \gets \Call{NumericalNoetherianOperators}{I,p}$
	    \State output $\gets$ append(output, $\{W, N\}$)
	\EndFor
    \State \textbf{return } output
\EndProcedure
\end{algorithmic}
\label{alg:num_primary_decomp}
\end{algorithm}

\begin{example}\label{example:part-of-normal-scroll}
    We compute a primary decomposition using our symbolic algorithm.
	Consider the rational normal scroll $S(2,2) \subseteq \mathbb{P}^5$ given by the prime ideal
	\begin{align*}
		P := I_2 \left( \begin{bmatrix}
			x_0 & x_1 & x_3 & x_4 \\
			x_1 & x_2 & x_4 & x_5
		\end{bmatrix} \right) \subseteq \KK[x_0, \dotsc, x_5]
	\end{align*}
	which has codimension 3 and degree 4. We can take $x_1,x_3,x_4$ as the dependent variables, and $x_0,x_2,x_5$ as independent variables.

	Consider the ideal $I$ generated by the following three polynomials:
	\begin{align*}
		f_1 &:= x_1^4-2 x_0 x_1^2 x_2+x_0^2 x_2^2+x_1 x_2 x_3 x_4-x_0 x_2 x_4^2-x_1^2 x_3 x_5+x_0 x_1 x_4 x_5\\
		f_2 &:= x_1^4-2 x_0 x_1^2 x_2+x_0^2 x_2^2+x_1 x_2 x_3 x_4-x_1^2 x_4^2-x_0 x_2 x_3 x_5+x_0 x_1 x_4 x_5\\
		f_3 &:= x_2^2 x_3 x_4-x_1 x_2 x_4^2+x_4^4-x_1 x_2 x_3 x_5+x_1^2 x_4 x_5-2 x_3 x_4^2 x_5+x_3^2 x_5^2
	\end{align*}
	This ideal was constructed to be a complete intersection defined by suitable linear combinations of generators of $P^2$. 
	Our goal is to compute a primary decomposition of $I$. Using Macaulay2 v1.15 on an Intel\textsuperscript{\textregistered} Core\textsuperscript{\texttrademark} i7-1065G7 CPU @ 1.30GHz, the command \texttt{primaryDecomposition I} did not terminate within 9 hours. On the other hand, \texttt{minimalPrimes I} quickly returns the primes
	\begin{align*}
	    P_1 &= (x_1,x_2,-x_4^2+x_3 x_5),\\
        P_2 &= (x_1,x_0,x_2^2 x_3 x_4+x_4^4-2 x_3 x_4^2 x_5+x_3^2 x_5^2),\\
        P_3 &= (x_4,x_3,-x_1^2+x_0 x_2),\\
        P_4 &= (x_4,x_5,-x_1^2+x_0 x_2),\\
        P_5 &= (-x_1^2+x_0 x_2,x_1 x_3-x_0 x_4,x_2 x_3-x_1 x_4,-x_2 x_4+x_1 x_5,-x_1 x_4+x_0 x_5,-x_4^2+x_3 x_5)
	\end{align*}

	Note that $P_5 = P$ is the prime ideal of the original rational normal scroll. The primes $P_i$ have dimension 3 and degrees (2, 4, 2, 2, 4) respectively.
	We then run \Cref{alg:symb_pos_dim} for the ideal $I$ and each minimal prime $P_i$. Noetherian operators for the $P_1$-primary component of $I$ are
	\begin{align*}
        N_{1,1} &= 1\\
        N_{1,2} &= \partial_{x_{4}}\\
        N_{1,3} &= \partial_{x_{1}}+\frac{1}{{x}_{3}} {x}_{4}\partial_{x_{2}}\\
        N_{1,4} &= \partial_{x_{1}}^{2}+\frac{2}{{x}_{3}} {x}_{4}\partial_{x_{1}}\partial_{x_{2}}+\frac{{x}_{5}}{{x}_{3}} \partial_{x_{2}}^{2}+\frac{2 {x}_{0}}{{x}_{3}^{2}} \partial_{x_{2}}\\
        N_{1,5} &= \partial_{x_{1}}^{3}+\frac{3}{{x}_{3}} {x}_{4}\partial_{x_{1}}^{2}\partial_{x_{2}}+\frac{3 {x}_{5}}{{x}_{3}} \partial_{x_{1}}\partial_{x_{2}}^{2}+\frac{{x}_{5}}{{x}_{3}^{2}} {x}_{4}\partial_{x_{2}}^{3}+\frac{6 {x}_{0}}{{x}_{3}^{2}} \partial_{x_{1}}\partial_{x_{2}}+\frac{6 {x}_{0}}{{x}_{3}^{3}} {x}_{4}\partial_{x_{2}}^{2}+\frac{12 {x}_{0}^{2}-6 {x}_{3}^{2}}{{x}_{3}^{4}{x}_{5}} {x}_{4}\partial_{x_{2}}\\
        N_{1,6} &= \partial_{x_{1}}^{4}+\frac{4}{{x}_{3}} {x}_{4}\partial_{x_{1}}^{3}\partial_{x_{2}}+\frac{6 {x}_{5}}{{x}_{3}} \partial_{x_{1}}^{2}\partial_{x_{2}}^{2}+\frac{4 {x}_{5}}{{x}_{3}^{2}} {x}_{4}\partial_{x_{1}}\partial_{x_{2}}^{3}+\frac{{x}_{5}^{2}}{{x}_{3}^{2}} \partial_{x_{2}}^{4}+\frac{12 {x}_{0}}{{x}_{3}^{2}} \partial_{x_{1}}^{2}\partial_{x_{2}}+\frac{24 {x}_{0}}{{x}_{3}^{3}} {x}_{4}\partial_{x_{1}}\partial_{x_{2}}^{2}\\
        &\quad+\frac{12 {x}_{0}{x}_{5}}{{x}_{3}^{3}} \partial_{x_{2}}^{3}+\frac{48 {x}_{0}^{2}-24 {x}_{3}^{2}}{{x}_{3}^{4}{x}_{5}} {x}_{4}\partial_{x_{1}}\partial_{x_{2}}+\frac{60 {x}_{0}^{2}-24 {x}_{3}^{2}}{{x}_{3}^{4}} \partial_{x_{2}}^{2}+\frac{-3 {x}_{0}^{2}}{{x}_{3}^{4}{x}_{5}} {x}_{4}\partial_{x_{4}}^{2}+\frac{120 {x}_{0}^{3}-48 {x}_{0}{x}_{3}^{2}}{{x}_{3}^{5}{x}_{5}} \partial_{x_{2}}
	\end{align*}
	
	For the $P_2$-primary component, we get Noetherian operators
	\begin{align*}
	    N_{2,1} &= 1
	\end{align*}

	For the $P_3$-primary component, we get Noetherian operators
	\begin{align*}
        N_{3,1} &= 1\\
        N_{3,2} &= \partial_{x_{3}}+\frac{1}{{x}_{0}} {x}_{1}\partial_{x_{4}}\\
        N_{3,3} &= \partial_{x_{1}}\\
        N_{3,4} &= \partial_{x_{1}}^{2}+\frac{4 {x}_{2}}{{x}_{5}^{2}} {x}_{1}\partial_{x_{3}}^{2}+\frac{8 {x}_{2}^{2}}{{x}_{5}^{2}} \partial_{x_{3}}\partial_{x_{4}}+\frac{4 {x}_{2}^{2}}{{x}_{0}{x}_{5}^{2}} {x}_{1}\partial_{x_{4}}^{2}+\frac{-8}{{x}_{0}{x}_{5}} {x}_{1}\partial_{x_{4}}
	\end{align*}

    For the $P_4$-primary component, we get Noetherian operators
	\begin{align*}
        N_{4,1} &= 1\\
        N_{4,2} &= \partial_{x_{4}}+\frac{1}{{x}_{0}} {x}_{1}\partial_{x_{5}}\\
        N_{4,3} &= \partial_{x_{1}}\\
        N_{4,4} &= \partial_{x_{1}}^{2}+\frac{4 {x}_{0}}{{x}_{3}^{2}} {x}_{1}\partial_{x_{4}}^{2}+\frac{8 {x}_{0}{x}_{2}}{{x}_{3}^{2}} \partial_{x_{4}}\partial_{x_{5}}+\frac{4 {x}_{2}}{{x}_{3}^{2}} {x}_{1}\partial_{x_{5}}^{2}+\frac{8}{{x}_{3}} \partial_{x_{5}}
	\end{align*}
		
	For the $P$-primary component, we get Noetherian operators
	\begin{align*}
        N_{5,1} &= 1\\
        N_{5,2} &= \partial_{x_{4}}\\
        N_{5,3} &= \partial_{x_{3}}\\
        N_{5,4} &= \partial_{x_{1}}\\
        N_{5,5} &= \partial_{x_{3}}\partial_{x_{4}}+\frac{{x}_{2}}{2 {x}_{0}{x}_{5}} {x}_{4}\partial_{x_{4}}^{2}\\
        N_{5,6} &= \partial_{x_{1}}\partial_{x_{3}}+\frac{{x}_{2}}{2 {x}_{5}} \partial_{x_{3}}^{2}+\frac{{x}_{2}}{2 {x}_{0}{x}_{5}} {x}_{4}\partial_{x_{1}}\partial_{x_{4}}\\
        N_{5,7} &= \partial_{x_{1}}^{2}+\left(\frac{4 {x}_{2}^{2}-2 {x}_{5}^{2}}{{x}_{5}^{3}} {x}_{4}+\frac{-16 {x}_{0}{x}_{2}^{2}+4 {x}_{0}{x}_{5}^{2}}{{x}_{2}^{3}}\right)\partial_{x_{3}}^{2}+\frac{4 {x}_{2}^{2}-{x}_{5}^{2}}{{x}_{2}^{2}} \partial_{x_{4}}^{2}\\
        N_{5,8} &= \partial_{x_{1}}^{3}+\frac{-12 {x}_{2}^{2}+6 {x}_{5}^{2}}{{x}_{2}{x}_{5}^{2}} {x}_{4}\partial_{x_{1}}^{2}\partial_{x_{3}}+\frac{-48 {x}_{0}{x}_{2}^{2}+12 {x}_{0}{x}_{5}^{2}}{{x}_{2}^{3}} \partial_{x_{1}}\partial_{x_{3}}^{2}\\
        &\quad+\left(\frac{64 {x}_{0}{x}_{2}^{4}-48 {x}_{0}{x}_{2}^{2}{x}_{5}^{2}+8 {x}_{0}{x}_{5}^{4}}{{x}_{2}^{4}{x}_{5}^{2}} {x}_{4}+\frac{-16 {x}_{0}{x}_{2}^{2}}{{x}_{5}^{3}}\right)\partial_{x_{3}}^{3}+\frac{-6 {x}_{2}^{2}+3 {x}_{5}^{2}}{{x}_{2}{x}_{5}} \partial_{x_{1}}^{2}\partial_{x_{4}}\\
        &\quad+\frac{-48 {x}_{2}^{2}+12 {x}_{5}^{2}}{{x}_{2}^{2}{x}_{5}} {x}_{4}\partial_{x_{1}}\partial_{x_{3}}\partial_{x_{4}}+\left(\frac{-24 {x}_{2}^{3}}{{x}_{5}^{4}} {x}_{4}+\frac{96 {x}_{0}{x}_{2}^{4}-72 {x}_{0}{x}_{2}^{2}{x}_{5}^{2}+12 {x}_{0}{x}_{5}^{4}}{{x}_{2}^{4}{x}_{5}}\right)\partial_{x_{3}}^{2}\partial_{x_{4}}\\
        &\quad+\frac{-12 {x}_{2}^{2}+3 {x}_{5}^{2}}{{x}_{2}^{2}} \partial_{x_{1}}\partial_{x_{4}}^{2}+\left(\frac{48 {x}_{2}^{4}-36 {x}_{2}^{2}{x}_{5}^{2}+6 {x}_{5}^{4}}{{x}_{2}^{3}{x}_{5}^{2}} {x}_{4}+\frac{-24 {x}_{2}^{3}}{{x}_{5}^{3}}\right)\partial_{x_{3}}\partial_{x_{4}}^{2}\\
        &\quad+\left(\frac{-8 {x}_{2}^{4}}{{x}_{0}{x}_{5}^{4}} {x}_{4}+\frac{8 {x}_{2}^{4}-6 {x}_{2}^{2}{x}_{5}^{2}+{x}_{5}^{4}}{{x}_{2}^{3}{x}_{5}}\right)\partial_{x_{4}}^{3}+\left(\frac{-48 {x}_{2}^{2}+12 {x}_{5}^{2}}{{x}_{2}^{3}{x}_{5}} {x}_{4}+\frac{48 {x}_{2}^{4}+6 {x}_{2}^{2}{x}_{5}^{2}-3 {x}_{5}^{4}}{{x}_{2}{x}_{5}^{4}}\right)\partial_{x_{3}}^{2}\\
        &\quad+\frac{12 {x}_{2}^{2}-3 {x}_{5}^{2}}{{x}_{0}{x}_{2}{x}_{5}^{2}} {x}_{4}\partial_{x_{1}}\partial_{x_{4}}+\frac{12 {x}_{2}^{2}-3 {x}_{5}^{2}}{{x}_{0}{x}_{2}^{2}{x}_{5}} {x}_{4}\partial_{x_{4}}^{2}
	\end{align*}
	
    From this we deduce that if $Q$ is the $P$-primary component of $I$, then the multiplicity of $Q$ over $P$ is 8 (note that this is consistent with the fact that $2 (6) + 4(1) + 2(4) + 2(4) + 4m(Q,P) = \deg I = 4^3$).
    Furthermore, as the set of Noetherian operators of $Q$ contains the set of Noetherian operators of $P^2$, namely $\{1, \partial_{x_1}, \partial_{x_3}, \partial_{x_4}\}$, we see that $Q$ is strictly contained in $P^2$.
    We can also see that the $P_2$-primary component is radical.

    One can also run \Cref{alg:num_primary_decomp} on this example: using a reasonable number of points quickly yields partial information about the Noetherian operators displayed above, such as the multiplicity. The caveat is that some of the rational functions have large degree (as evidenced by the above), so interpolating those coefficients will take correspondingly longer times. 
\end{example}

\begin{example} \label{ex:carpet}
Next, we illustrate a numerical primary decomposition using \Cref{alg:mainAlgorithm}. 
Let $J$ be the ideal of the K3 carpet over the scroll $S(3,3) \subseteq \mathbb{P}^7$, i.e.
\begin{align*}
    J &:= (x_1^2-x_0 x_2,x_1 x_2-x_0 x_3,x_2^2-x_1 x_3,x_2 y_0-2 x_1 y_1+x_0 y_2,x_3 y_0-2 x_2 y_1+x_1 y_2,\\
&\qquad x_2 y_1-2 x_1 y_2+x_0 y_3,x_3 y_1-2 x_2 y_2+x_1 y_3,y_1^2-y_0 y_2,y_1 y_2-y_0 y_3,y_2^2-y_1 y_3)
\end{align*}
in the ring $\mathbb{Q}[x_0,\dotsc,x_3, y_0,\dotsc,y_3]$. 
Let $I$ be the ideal of a generic complete intersection of quadrics containing the carpet, generated by 5 random $\mathbb{Q}$-linear combinations of the 10 generators of $J$.

Neither \texttt{primaryDecomposition I} nor \texttt{minimalPrimes I} terminated within 9 hours.
However, a numerical irreducible decomposition reveals that $I$ has two minimal primes, of dimension 3 and degrees (6, 20) respectively.
We then run \Cref{alg:num_primary_decomp} on the witness sets.

Let $Q$ be the component primary to the degree 6 minimal prime of $I$.
We obtain
\begin{align*}
    N_{Q,1} &= 1\\
    N_{Q,2} &= \partial_{y_{0}}+\frac{{.666667} {x}_{1}}{{x}_{0}} \partial_{y_{1}}+\frac{{.333333} {x}_{2}}{{x}_{0}} \partial_{y_{2}}
\end{align*}
as Noetherian operators for $Q$. 
The component primary to the degree 20 minimal prime (which defines a generic link of the K3 carpet) has Noetherian operators $\{1\}$, i.e. is radical. For timing: the numerical irreducible decomposition took under 3 seconds, and computing Noetherian operators took under 2 seconds for each component.

As the degree 6 minimal prime obtained from the numerical irreducible decomposition is the scroll $S(3,3)$ (being of minimal degree), i.e. $\sqrt{Q} = \sqrt{J}$, a natural question that arises is whether $Q$ is in fact equal to $J$.
We may verify this by directly computing Noetherian operators of $J$ using \Cref{alg:symb_pos_dim}, obtaining
\begin{align*}
    N_{J,1} &= 1\\
    N_{J,2} &= \partial_{y_0} + \frac{2x_3 y_1}{3x_0 y_3} \partial_{y_1} + \frac{x_3 y_2}{3 x_0 y_3} \partial_{y_2}
\end{align*}
Although the Noetherian operators for $J$ and the Noetherian operators for $Q$ look different, the rational function coefficients are equal on the minimal prime of interest, which is the scroll (note that $x_3 y_1 - x_1 y_3$ and $x_3 y_2 - y_3 x_2$ both lie in $\sqrt{J}$).
This confirms that $Q = J$.
In this case, the numerical interpolation found representatives of rational functions with lowest possible degrees.

\end{example}

\section{Non-primary ideals} \label{sec:gen_properties}
Thus far, we have focused our attention on primary ideals. 
As we have seen, this is enough for the purpose of numerical primary decomposition, cf. \Cref{alg:num_primary_decomp}. 
In this last section, we discuss some properties and behaviors of Noetherian operators for arbitrary (i.e. not necessarily primary) ideals.
First, we record how Noetherian operators vary under linear coordinate changes.

\begin{proposition}
Let $R := \KK[x_1, \ldots, x_n] = \KK[\xx]$, and let $\varphi$ be a $\KK$-linear automorphism of $R$ given by $\varphi(\xx) := A \xx$ for some $A \in GL_n(\KK)$. 
Define a $\KK$-linear automorphism of the Weyl algebra $W_R = \KK[\xx]\langle \partial \rangle$ by 
\[
\psi : \begin{pmatrix}
\xx \\
\partial 
\end{pmatrix} \mapsto 
\begin{pmatrix}
A \xx \\
(A^{-1})^T \partial 
\end{pmatrix}.
\]
If $I \subseteq R$ is an ideal, and $D_1, \ldots, D_r$ is a set of Noetherian operators for $I$, then $\psi(D_1), \ldots, \psi(D_r)$ is a set of Noetherian operators for $\varphi(I) \subseteq R$.
\end{proposition}

\begin{proof}
For $f \in R$, one has
\begin{align*}
f \in \varphi(I) &\iff \varphi^{-1}(f) \in I \iff D_i \bullet \varphi^{-1}(f) \in \sqrt{I} \quad \forall i = 1, \ldots, r \\
&\iff \varphi(D_i \bullet \varphi^{-1}(f)) \in \sqrt{\varphi(I)} \quad \forall i = 1, \ldots, r,
\end{align*}
since $\sqrt{\varphi(I)} = \varphi(\sqrt{I})$, as $\varphi$ is a $\KK$-linear automorphism of $R$. 
Writing $D_i = \sum_\alpha p_\alpha \partial^\alpha$, we have $\varphi(D_i \bullet \varphi^{-1}(f)) = \varphi( (\sum_\alpha p_\alpha \partial^\alpha) \bullet \varphi^{-1}(f)) = \sum_\alpha \varphi(p_\alpha) \varphi(\partial^\alpha \bullet \varphi^{-1}(f))$, so it suffices to show that $\varphi(\partial^\alpha \bullet \varphi^{-1}(f)) = \psi(\partial^\alpha) \bullet f$ for any $f \in R$. 
By linearity, it suffices to check this when $f = \xx^\beta$ is a monomial, i.e. we must show $\varphi(\partial^\alpha \bullet \varphi^{-1}(\xx^\beta)) = \psi(\partial^\alpha) \bullet \xx^\beta$ for all $\alpha, \beta \in \mathbb{N}^n$. 

We first consider the case where $\alpha, \beta$ are standard basis vectors, i.e. $\partial^\alpha = \partial_{x_j}$ and $\xx^\beta = x_i$ for some $i, j \in \{1, \ldots, n\}$. 
Then $\varphi \left(\partial_{x_j} \bullet \varphi^{-1}(x_i) \right) = \varphi \left(\partial_{x_j} \bullet \sum_{k=1}^n (A^{-1})_{i,k} x_k \right) = \varphi \left((A^{-1})_{i,j} \right) = (A^{-1})_{i,j} = \left( \sum_{k=1}^n (A^{-1})_{k,j} \partial_{x_k} \right) \bullet x_i = \psi(\partial_{x_j}) \bullet x_i$.

To show that this extends to arbitrary $\beta$, note that both $\varphi \left( \partial_{x_j} \bullet \varphi^{-1}( \rule{0.2cm}{0.15mm}) \right)$ and $\psi(\partial_{x_j}) \bullet ( \rule{0.2cm}{0.15mm})$ are both differential operators, which must satisfy the product rule, so if these agree on every variable $x_i$ then they agree on every monomial $\xx^\beta$. 
To extend to arbitrary $\alpha$, note that $\psi$ preserves multiplication in $W$ by definition, so
\begin{align*}
\varphi \left( \partial_{x_j} \partial_{x_k} \bullet \varphi^{-1}(\rule{0.2cm}{0.15mm}) \right) &= 
\varphi \left( \partial_{x_j} \bullet \varphi^{-1} \varphi \left( \partial_{x_k} \bullet \varphi^{-1}(\rule{0.2cm}{0.15mm}) \right) \right) \\
&= \varphi \left( \partial_{x_j} \bullet \varphi^{-1} (\psi(\partial_{x_k}) \bullet (\rule{0.2cm}{0.15mm}) ) \right) \\
&= \psi(\partial_{x_j}) \bullet \psi(\partial_{x_k} \bullet ( \rule{0.2cm}{0.15mm} )) \\
&= \psi(\partial_{x_j}) \psi(\partial_{x_k} \bullet ( \rule{0.2cm}{0.15mm} )) \\
&= \psi(\partial_{x_j} \partial_{x_k}) \bullet ( \rule{0.2cm}{0.15mm} )
\end{align*}
hence inductively $\varphi \left( \partial^\alpha \bullet \varphi^{-1}(\rule{0.2cm}{0.15mm}) \right) = \psi(\partial^\alpha) \bullet (\rule{0.2cm}{0.15mm})$ for any $\alpha$.
\end{proof}

Next, we give a construction for a global set of Noetherian operators for an unmixed ideal:

\begin{proposition}
Let $I$ be an unmixed ideal, with a minimal primary decomposition $I = q_1 \cap \ldots \cap q_r$, and let $N_i$ be a set of Noetherian operators for $q_i$ for $i = 1, \ldots, r$. 
For $D \in \bigcup_i N_i$, choose $\displaystyle h_D \in \bigcap_{D \not \in N_j} \sqrt{q_j} \setminus \bigcup_{D \in N_i} \sqrt{q_i}$.
Then $N := \{ h_D D \mid D \in \bigcup_i N_i \}$ is a set of Noetherian operators for $I$.
\end{proposition}

\begin{proof}
First, note that if $\displaystyle \bigcap_{D \not \in N_j} \sqrt{q_j} \subseteq \bigcup_{D \in N_i} \sqrt{q_i}$ for some $D$, then $\displaystyle \bigcap_{D \not \in N_j} \sqrt{q_j} \subseteq \sqrt{q_i}$ for some $i$ by prime avoidance, and then $\sqrt{q_j} \subseteq \sqrt{q_i}$ for some $i \ne j$, contradicting the unmixedness assumption on $I$. 
Thus choices of $h_D$ always exist.

Suppose $f \in I$, and choose $D \in \bigcup_i N_i$. 
For any $i$ with $D \in N_i$, we have $f \in q_i \implies D \bullet f \in \sqrt{q_i}$. 
By choice of $h_D$, this implies $\displaystyle h_D D \bullet f \in \left( \bigcap_{D \in N_i} \sqrt{q_i} \right) \cap \left( \bigcap_{D \not \in N_j} \sqrt{q_j} \right) = \sqrt{I}$.

Conversely, suppose $f \not \in I$. 
Then WLOG $f \not \in q_1$, so there exists $D_1 \in N_1$ such that $D_1 \bullet f \not \in \sqrt{q_1}$. 
Since also $h_{D_1} \not \in \sqrt{q_1}$ and $\sqrt{q_1}$ is prime, this means $h_{D_1} D_1 \bullet f \not \in \sqrt{q_1}$, and thus $h_{D_1} D_1 \bullet f \not \in \sqrt{I}$.
\end{proof}

Finally, we consider the question of recovering $\sqrt{I}$ from the data of $I$ and Noetherian operators for $I$.
Fix a finite generating set $G$ of $I$ and a set of Noetherian operators $N$ of $I$.
We consider the ideal $N(G) := (D \bullet g \mid D \in N, \, g \in G)$ obtained by applying operators in $N$ to the generating set $G$. 
Note that since $G$ generates $I$, one has $N(G) = (D \bullet f \mid D \in N, \, f \in I)$ -- in particular, $N(G)$ does not depend on the choice of $G$, and one always has $N(G) \subseteq \sqrt{I}$ by definition.
However, even if $I$ is primary, $N(G)$ need not equal $\sqrt{I}$:

\begin{example} \label{ex:recoveryNotUnmixed}
Let $I = ((xy - z^2)^2) \subseteq \mathbb{C}[x,y,z]$. 
Then $N = \{1, \partial_y\}$ is a set of Noetherian operators of $I$. 
Applying $N$ to the single generator of $I$ yields $N(G) = ((xy-z^2)^2, 2x(xy-z^2))$, which is strictly contained in $\sqrt{I} = (xy - z^2)$.
\end{example}

However, the issue in \Cref{ex:recoveryNotUnmixed} was that $N(G)$ was not unmixed (whereas radical ideals are evidently unmixed), which turns out to be the only obstruction:

\begin{proposition}
If $I = (G)$ is primary, and $N$ is a set of Noetherian operators for $I$ constructed as in \Cref{prop:pos_dim_noeth_ops}(1), then the unmixed part of $N(G)$ is $\sqrt{I}$.
\end{proposition}
\begin{proof}
Let $P = \sqrt{I}$. 
Since WLOG $1$ is in the $\KK$-span of $N$, we have $I \subseteq N(G) \subseteq P$, which implies that $\sqrt{N(G)} = P$. 
Let $Q$ be the unmixed part of $N(G)$, which is the $P$-primary component of $N(G)$.

First consider the case $\dim I = 0$, so that $Q = N(G)$, the dual space $D_P[P]$ is spanned by $\{1\}$, and $N \subseteq D_P[I]$. 
Suppose that $Q \neq P$, so that $\dim_{\kappa(P)} D_P[Q] > 1$, hence $D_P[Q]$ contains a nonzero element $p$ of $\partial$-degree $\ge 1$.  
For each $D \in N$, the operator $p\circ D$ is an element of $D_P[I]$ (since $D \bullet f \in Q$ for all $f \in I$).
Choosing some $D \in N$ of maximal degree gives that $p \circ D$ is outside the linear span of $N$. 
Therefore $D_P[I]$ is strictly larger than the span of $N$, contradicting \Cref{cor:minimalNops}.

If now $I$ is primary of any dimension, we may perform Noether normalization to obtain a zero-dimensional ideal $IS$. 
By \Cref{prop:pos_dim_noeth_ops}(2), the set $N$ gives a set of Noetherian operators for $IS$. 
Then the reasoning in the zero-dimensional case above shows that $QS = PS$, which implies $Q = P$.
\end{proof}

\bibliography{references}{}
\bibliographystyle{plain}

\end{document}